\documentclass{amsart} 
\usepackage{amscd,amssymb,latexsym,subfigure,verbatim,hyperref,epsfig,amsmath,supertabular}
\usepackage[graph,frame,poly,arc]{xy}
\usepackage{calc}
\usepackage{amsmath}
\usepackage{graphicx}

\begin{document}

\newcommand{\mmbox}[1]{\mbox{${#1}$}}
\newcommand{\proj}[1]{\mmbox{{\mathbb P}^{#1}}}
\newcommand{\affine}[1]{\mmbox{{\mathbb A}^{#1}}}
\newcommand{\Ann}[1]{\mmbox{{\rm Ann}({#1})}}
\newcommand{\caps}[3]{\mmbox{{#1}_{#2} \cap \ldots \cap {#1}_{#3}}}
\newcommand{\N}{{\mathbb N}}
\newcommand{\Z}{{\mathbb Z}}
\newcommand{\R}{{\mathbb R}}
\newcommand{\K}{{\mathbb K}}
\newcommand{\p}{{\mathbb P}}
\newcommand{\A}{{\mathcal A}}
\newcommand{\RJ}{{\mathcal R}/{\mathcal J}}
\newcommand{\C}{{\mathbb C}}
\newcommand{\CR}{C^r(\hat P)}
\newcommand{\Der}{\mathop{\rm Der}\nolimits}
\newcommand{\Tor}{\mathop{\rm Tor}\nolimits}
\newcommand{\ann}{\mathop{\rm ann}\nolimits}
\newcommand{\Ext}{\mathop{\rm Ext}\nolimits}
\newcommand{\Hom}{\mathop{\rm Hom}\nolimits}
\newcommand{\im}{\mathop{\rm Im}\nolimits}
\newcommand{\rank}{\mathop{\rm rank}\nolimits}
\newcommand{\supp}{\mathop{\rm supp}\nolimits}
\newcommand{\arrow}[1]{\stackrel{#1}{\longrightarrow}}
\newcommand{\coker}{\mathop{\rm coker}\nolimits}
\sloppy
\newtheorem{defn0}{Definition}[section]
\newtheorem{prop0}[defn0]{Proposition}
\newtheorem{conj0}[defn0]{Conjecture}
\newtheorem{thm0}[defn0]{Theorem}
\newtheorem{lem0}[defn0]{Lemma}
\newtheorem{rmk0}[defn0]{Remark}
\newtheorem{corollary0}[defn0]{Corollary}
\newtheorem{example0}[defn0]{Example}
\newcommand{\An}{\mathop{\rm A_n}\nolimits}
\newcommand{\conv}{\mathop{conv}\nolimits}
\newcommand{\Aroot}{\mathop{\rm A}\nolimits}
\newenvironment{defn}{\begin{defn0}}{\end{defn0}}
\newenvironment{prop}{\begin{prop0}}{\end{prop0}}
\newenvironment{conj}{\begin{conj0}}{\end{conj0}}
\newenvironment{thm}{\begin{thm0}}{\end{thm0}}
\newenvironment{lem}{\begin{lem0}}{\end{lem0}}
\newenvironment{cor}{\begin{corollary0}}{\end{corollary0}}
\newenvironment{exm}{\begin{example0}\rm}{\end{example0}}
\newenvironment{rmk}{\begin{rmk0}\rm}{\end{rmk0}}

\newcommand{\defref}[1]{Definition~\ref{#1}}
\newcommand{\propref}[1]{Proposition~\ref{#1}}
\newcommand{\thmref}[1]{Theorem~\ref{#1}}
\newcommand{\lemref}[1]{Lemma~\ref{#1}}
\newcommand{\corref}[1]{Corollary~\ref{#1}}
\newcommand{\exref}[1]{Example~\ref{#1}}
\newcommand{\secref}[1]{Section~\ref{#1}}
\newcommand{\poina}{\pi({\mathcal A}, t)}
\newcommand{\poinc}{\pi({\mathcal C}, t)}
\newcommand{\std}{Gr\"{o}bner}
\newcommand{\jq}{J_{Q}}

\title {Splines on the Alfeld split of a simplex\\
and type A root systems}

\author{Hal Schenck}
\thanks{Schenck supported by  NSF 1068754, NSF 1312071}\address{Schenck: Mathematics Department \\ University of Illinois Urbana-Champaign\\
  Urbana \\ IL 61801\\ USA}
\email{schenck@math.uiuc.edu}


\begin{abstract}
\noindent In \cite{a1}, Alfeld introduced a tetrahedral analog 
$AS(\Delta_3)$ of the Clough-Tocher split of a triangle. A formula for
the dimension of the spline space $C^r_k(AS(\Delta_n))$ is 
conjectured in \cite{fs}. We prove that the graded module 
of $C^r$-splines on the cone over $AS(\Delta_n)$ is isomorphic to the module 
$D^{r+1}(\mbox{A}_n)$ of multiderivations on the type A$_n$ Coxeter arrangement. A theorem of Terao shows that the module of 
multiderivations of a Coxeter arrangement is free and gives
an explicit basis. As a consequence the conjectured formula holds.
\end{abstract}
\maketitle
\vskip -.1in
\vskip -.1in
\section{Introduction}\label{sec:intro}
The Alfeld split $AS(\Delta_n)$ of the $n$-simplex $\Delta_n$ is obtained by adding a single central interior vertex, and then coning over the boundary 
of $\Delta_n$. Hence for $i>0$, every $i+1$-face of $AS(\Delta_n)$ 
corresponds to an $i$-face of $\partial(\Delta_n)$, and $AS(\Delta_n)$
has a single interior vertex. This construction was introduced 
in \cite{a1} as a tetrahedral analog of the Clough-Tocher split of a triangle,
and \cite{fs} makes the following
\begin{conj}$[$Foucart-Sorokina$]$
The dimension of the space $C^r_k(AS(\Delta_n))$ of splines of
degree $\le k$ in $n$-variables over the Alfeld split of $\Delta_n$
is given by 
\[ 
\dim C^r_k(AS(\Delta_n)) = {k+n \choose n} +  \begin{cases} 
n {k+n-\frac{(r+1)(n+1)}{2} \choose n}, \mbox{ if r is odd}\\
{k+n-1-\frac{r(n+1)}{2} \choose n}+\cdots + {k-\frac{r(n+1)}{2} \choose n}, \mbox{ if r is even}.
\end{cases}
\] 
\end{conj}
In this brief note, we prove the conjecture. The proof hinges on a
connection between splines and the theory of hyperplane arrangements.

In \cite{b}, Billera showed that splines on a simplicial complex $\Delta$ 
can be obtained as the top homology 
module of a chain complex, and used this to solve a longstanding 
conjecture of Gil Strang. 
Billera-Rose show in \cite{br} that by forming the cone $\hat{\Delta} \subseteq {\mathbb R}^{n+1}$ over $\Delta \subseteq \mathbb{R}^n$, commutative algebra can 
be used to study the dimension of $C^r_k(\Delta)$: $C^r(\hat \Delta)$ is a graded module 
over the polynomial ring $S=\mathbb{R}[x_0,\ldots,x_n]$, and the dimension
of $C^r_k(\Delta)$ is the dimension of the $k^{th}$ graded piece of 
$C^r(\hat \Delta)$. 

A spectral sequence argument in \cite{s1} 
shows that $C^r(\hat \Delta)$ is a free module iff all the lower homology 
modules of a certain chain complex vanish, and Corollary 4.12 of \cite{s1}
shows that in this case the dimension of $C^r_k(\Delta)$ is determined 
entirely by local data. In \cite{gs}, the algebraic geometry of fatpoints 
in projective space is used to study the dimension of quotients of 
ideals generated by powers of linear forms. Such modules appear in the 
chain complex whose top homology module is $C^r(\hat \Delta)$. In \S 2 we 
give more details about these methods, and in \S 3 we show 
Conjecture 1.1 follows from a theorem of Terao \cite{t}. The case $r=1$
was recently proved in \cite{ks}.
\section{Algebraic preliminaries}
Throughout the paper, our references are de Boor \cite{DB} for splines and 
Eisenbud \cite{E} for commutative algebra.
\begin{defn}\label{compC}\cite{s1} For a simplicial complex $\Delta \subseteq \R^n$, 
let $\RJ(\Delta)$ be the complex of $S=\R[x_0,\ldots,x_n]$ modules, 
with differential $\partial_i$ the usual boundary operator in relative
(modulo boundary) homology.
\[
0 \longrightarrow \bigoplus\limits_{\sigma \in \Delta_n} S 
\stackrel{\partial_n}{\longrightarrow} \bigoplus\limits_{\tau \in \Delta_{n-1}^0} S/I_{\tau}  
\stackrel{\partial_{n-1}}{\longrightarrow} \bigoplus\limits_{\psi \in \Delta_{n-2}^0} S/I_{\psi}  
\stackrel{\partial_{n-2}}{\longrightarrow} \ldots
\stackrel{\partial_{1}}{\longrightarrow} \bigoplus\limits_{v \in \Delta_{0}^0} S/I_{v}\longrightarrow 0,
\]
where for an interior $i$-face $\gamma \in \Delta_i^0$, we define 
\[
I_\gamma = \langle l_{\hat \tau}^{r+1} \mid \hat \gamma \subseteq \hat \tau \in \hat \Delta_{n-1} \rangle.
\]
\end{defn}
This differs from the complex in \cite{b}: 
the ideal $I_\gamma$ is generated by $r+1^{st}$ powers of 
(homogenizations of) linear forms whose vanishing defines the hyperplanes 
$\tau$ which contain $\gamma$. As with the complex in \cite{b}, the top 
homology module of $\RJ(\Delta)$ computes splines of smoothness $r$ on $\hat{\Delta}$.
Theorem 4.10 of \cite{s1} shows that if $\Delta$ is a 
topological $n$-ball, then the module
$C^r(\hat \Delta)$ is free iff $H_i(\RJ(\Delta))=0$ for all $i < n$. 
Computing the dimension of the modules which appear in 
the complex $\RJ(\Delta)$ is nontrivial as soon
as $n \ge 3$. Indeed, for ideals generated by powers of linear forms
in three variables, no answer is known, even if the forms are in
general position. This is due to a classical connection from algebraic
geometry known as Macaulay inverse systems, which translates questions
about powers of linear forms into questions about fatpoints, which are
points, with multiplicity, in projective space. For more details on
inverse systems, see \cite{gs}. In fact, for powers of linear forms in
three varibles, inverse systems translate to questions about fatpoints
in $\p^2$, which is the topic of the famous (and open) 
Segre-Harbourne-Gimigliano-Hirschowitz conjecture \cite{m}.

In the case of the Alfeld split of $\Delta_n$, it is easy
to see that at every interior $i$-face there are exactly $n-i+1$ forms 
incident to the face. This means that after a change of variables,
we can assume that the ideals which appear are of the form
\begin{equation}\label{local}
I_\gamma = \langle x_1^{r+1},\ldots,  x_{n-i}^{r+1}, (x_1+\cdots+x_{n-i})^{r+1} \rangle
\end{equation}
For these ideals, the corresponding fatpoint scheme consists
of the coordinate points and the point $(1:1:\cdots :1)$, and the dimension
of $I_{\gamma}$ in any degree is as expected, in the sense of \cite{m}. 
As a consequence, we have
\begin{lem}
Conjecture 1.1 holds if $C^r(\widehat{AS(\Delta_n)})$ is a free $S$-module.
\end{lem}
\begin{proof}
By Theorem 4.10 of \cite{s1}, $C^r(\widehat{AS(\Delta_n)})$ is free iff 
$H_i(\mathcal{R/J})=0$ for all $i < n$, and in this case, 
Corollary 4.12 of \cite{s1} shows that the dimension of 
$C^r_k(\widehat{AS(\Delta_n)})$ is the alternating sum of the dimensions 
of the degree $k$ graded components of the modules appearing in
the chain complex $\mathcal{R/J}$. These are the 
dimensions of the ideals in Equation (1), which 
can be worked out using inverse systems.
\end{proof}
Conjecture 1.1 could hold even if  $C^r(\widehat{AS(\Delta_n)})$ is not free: 
contributions from nonzero $H_i(\mathcal{R/J})$ could cancel out. 
However, in the next section, we will see that Conjecture 1.1 
is a consequence of a stronger fact: $C^r(\widehat{AS(\Delta_n)})$ is a free module.
\section{Conjecture 1.1 via Arrangement Theory}
In \cite{br}, Billera-Rose note that $C^r(\hat \Delta)$ may be computed
as the kernel of the matrix
\[
\phi=\left[ \partial_n \Biggm| \begin{array}{*{3}c}
l_{\tau_1}^{r+1} & \  & \  \\
\ & \ddots & \  \\
\ & \ & l_{\tau_m}^{r+1}
\end{array} \right], 
\]
where $\partial_n$ is the $n^{th}$ simplicial boundary map in relative 
homology; each row encodes the smoothness condition
across a codimension one face $\tau_i$. The key observation needed for the proof 
of Conjecture 1.1 relates the matrix above to a matrix which appears in
hyperplane arrangement theory, and is motivated by a result of \cite{s2},
where it appears in a different guise. First, we need some definitions. 
A hyperplane arrangement is a collection of hyperplanes 
\[
\A = \bigcup\limits_{i=1}^dH_i \subseteq \K^n,
\]
where $\K$ is typically $\R$ or $\C$. An arrangement is {\em central}
if all hyperplanes $H_i$ pass thru the origin, which is true iff for 
every $H_i=V(l_i)$, the linear form $l_i$ is homogeneous. Given a simplicial complex 
$\Delta \subseteq \R^n$, by construction the codimension one faces of
the cone $\hat \Delta$ yield a central arrangement. 
\begin{rmk}
Embedding $\Delta \subseteq \R^n$ in the hyperplane 
$x_{n+1} =1 \subseteq \R^{n+1}$ and coning with the origin yields
the complex $\widehat{\Delta}$. The set of 
splines on $\widehat{\Delta}$ has the additional structure of
being a graded $\R[x_0,\ldots, x_n]$ module. Homogenizing 
splines in $C^r_k(\Delta)$ using the variable $x_{n+1}$ yields
the formula $\dim C^r_k(\Delta) = \dim C^r(\widehat{\Delta})_k$.
The point is that coning is necessary in order to apply 
the tools of commutative algebra.
\end{rmk}
One of the main algebraic objects associated to a central arrangement
is the graded $S$--module $D({\mathcal A})$ of vector fields
tangent to $\A$. Recall $\Der_{\R}(S)$ is the free $S$--module with
basis $\partial/\partial(x_i)$. 
\begin{defn}
The module of derivations of $\A$ is 
\[
D({\mathcal A}) = \{ \theta \mid \theta(l_i) \in \langle l_i \rangle \mbox{ for all }V(\l_i) \in \A \} \subseteq Der_{\R}(S),\]
and the module of m-multiderivations of $\A$ is 
\[
D^m({\mathcal A}) = \{ \theta \mid \theta(l_i) \in \langle l_i^m \rangle \mbox{ for all }V(\l_i) \in \A \} \subseteq Der_{\R}(S).
\]
An arrangement $\mathcal{A}$ is m-multifree if $D^m({\mathcal A}) \simeq
\oplus S(-d_i)$. In this case the $d_i$ are called the 
m-multiexponents of ${\mathcal A}$.
\end{defn}
It follows from the definition that the m-multiderivations can be computed as the 
kernel of a matrix similar to $\phi$ above: $D^m(\A)$ is the kernel of
\[
{\small \psi = \left[ \delta_n \Biggm| \begin{array}{*{3}c}
l^m_1 & \  & \  \\
\ & \ddots & \  \\
\ & \ & l^m_d
\end{array} \right],}
\]
where $\delta_n$ is a $d \times n$ matrix, such that if $l_j =\sum a^j_i x_i$, 
then the $j^{th}$ row of $\delta_n$ is $[a^j_1,\ldots, a^j_d]$. 
To relate $AS(\Delta_n)$ to arrangements, we need the following:
\begin{defn}\label{AN}
The reflection arrangement of type A$_n$ is defined by the
reflecting hyperplanes of $SL(n)$, which are given by 
$\bigcup_{1 \le i < j \le n+1}V(x_j-x_i) = \An \subseteq \R^{n+1}$.
\end{defn}

\begin{thm}[Terao, \cite{t}]
$D^m(\mbox{A}_n)$ is m-multifree for all $m$, with exponents 
\[ 
\begin{cases} 
\frac{m(n+1)}{2}, \ldots,\frac{m(n+1)}{2} \mbox{ if m is even}\\
\frac{(m-1)(n+1)}{2}+1, \cdots, \frac{(m-1)(n+1)}{2}+n, \mbox{ if m is odd}.
\end{cases}
\] 
\end{thm}
Terao proves the theorem for all reflection groups, not just
those of type A. For the free $S = \R[x_0,\ldots,x_n]$ module 
$S(-a)$ generated in degree $a$, 
\[
\dim_{\R}S(-a)_k = {n-a+k \choose n}.
\]
Terao's formula does not include the constant derivation, 
which corresponds to the constant splines contributing the
term ${k+n \choose n}$ in Conjecture 1.1. So it follows 
from Terao's theorem that the dimension of $D^{r+1}(\mbox{A}_n)_k$ 
is equal to the formula of Conjecture 1.1. Before we prove
the theorem, we give a simple motivating example.
\begin{exm}
Let $\Delta_2$ have vertices $\{{\bf e}_0=(0,0),{\bf e}_1=(1,2),{\bf e}_2=(2,1)\}$, and place
a central vertex at $(1,1)$. Order the triangles
of $\Delta_2$ as $[{\bf e}_0, {\bf e}_1],[{\bf e}_0, {\bf e}_2],[{\bf e}_1, {\bf e}_2]$ and the rays as ${\bf e}_0, {\bf e}_1, {\bf e}_2$. Note that
the linear forms vanishing on the rays are $x_1-x_2, x_1-1, x_2-1$.
Homogenizing with respect to $x_3$, we obtain
\[
\phi = {\small \left[ \begin{array}{*{3}c}
-1 & 1&0  \\
1 & 0 &-1\\
0& -1 &1
\end{array}
\begin{array}{*{3}c}
(x_1-x_2)^{r+1} &  0  & 0\\
 0   & (x_1-x_3)^{r+1}  &0 \\
 0   &0   &(x_2-x_3)^{r+1} 
\end{array} \right]}
\]
On the other hand, the matrix computing  $D^{r+1}(\mbox{A}_2)$is
\[
\psi= {\small \left[ \begin{array}{*{3}c}
1 & -1&0  \\
1 & 0 &-1\\
0& 1 &-1
\end{array}
\begin{array}{*{3}c}
(x_1-x_2)^{r+1} &  0  & 0\\
 0   & (x_1-x_3)^{r+1}  &0 \\
 0   &0   &(x_2-x_3)^{r+1} 
\end{array} \right]}
\]
The signs in first and third rows of the matrices differ, but this can be
fixed by reversing the orientation of ${\bf e}_0$ and ${\bf e}_2$. 
So for the choice of vertices above 
\[
C^r(\widehat{AS(\Delta_2)}) \simeq D^{r+1}(\mbox{A}_2).
\]
\end{exm}
\begin{thm}\label{Aroot}
$C^r(\widehat{AS(\Delta_n)}) \simeq D^{r+1}(\An)$. 
\end{thm}
\begin{proof}
In the matrix for $\psi$, the rows are indexed by pairs $\{i,j\}$ with 
$1 \le i < j \le n+1$, and the columns are indexed
by $m \in \{1,\ldots, n+1\}$. In the row of $\delta_n$ indexed by $\{i,j\}$,
there is a $+1$ in column $i$ and a $-1$ in column $j$.

Choose vertices for $AS(\Delta_n)$ as follows: for $i \in \{1,\ldots,n\}$, 
let $\epsilon_i$ be the $i^{th}$ standard basis vector of $\R^n$, and define 
\[
\begin{array}{cccc}
{\bf v} &= &\sum_{i=1}^n \epsilon_i &\\
{\bf e}_0 &= & {\bf 0}&\\
{\bf e}_i &= & {\bf v} + \epsilon_i,& i \in \{1,\ldots,n\}
\end{array}
\]
Let $\Delta_n =\conv\{{\bf e}_0, \ldots, {\bf e}_n \}$,
and let ${\bf v}$ be the interior vertex of $AS(\Delta_n)$. 
In the matrix $\phi$, the component $\partial_n$ also has a single $+1$ and $-1$ in
each row. The columns are indexed by the maximal faces of $\Delta_n$.
Since $\Delta_n$ has $n+1$ vertices, each maximal face is determined
by leaving out a single vertex. The rows of $\phi$ are indexed by the 
codimension one faces, which are determined by leaving out a 
pair $\{i,j\}$ of vertices. 

Index the columns as 
$\overline{{\bf e}_n}, \ldots, \overline{{\bf e}_1}, \overline{{\bf e}_0}$,
and the rows as $\overline{{\bf e}_n{\bf e}_{n-1}}, \ldots, \overline{{\bf e}_1{\bf e}_0}$.
Then in the row indexed by $\overline{{\bf e}_i{\bf e}_{j}}$ with $i<j$,
there will be a $+1$ in the column indexed by $\overline{{\bf e}_j}$, and 
a $-1$ in the column indexed by $\overline{{\bf e}_i}$, up to multiplying
by $-1$, which does not change the smoothness condition across faces.
In particular, the $\partial_n$ block of $\phi$ agrees with the $\delta_n$
block of $\psi$. We need only check that the equations of the hyperplanes 
agree, which follows from our choice of vertex locations of 
$AS(\Delta)_n$. To be explicit, let 
\[
M = {\small \left[ \begin{array}{*{6}c}
x_1 &  \vdots       &           &           &        &  \vdots         \\
x_2 & {\bf v} & {\bf e}_0 & {\bf e}_1 & \cdots &  {\bf e}_n\\
\vdots & \vdots      &          &           &        &   \vdots         \\
x_{n+1}& 1     &1         & 1         & \cdots & 1
\end{array} \right]}
\]
If $L_{ij}$ is the facet which misses ${\bf e}_i$ and ${\bf e}_j$
and $M_{ij}$ is the matrix obtained from $M$ by deleting the columns
containing ${\bf e}_i$ and ${\bf e}_j$, then then $L_{ij}= \det(M_{ij})$.
Since any two sets of $n+2$ points in general position in $\p^n$ are 
projectively equivalent, the result is in fact independent of the
choice of ${\bf v}, {\bf e}_0, \ldots, {\bf e}_n$.
\end{proof}
\begin{exm} In the case of Example 3.5, the matrix M is 
\[
{\small \left[ \begin{array}{*{5}c}
x_1 & 1       &0          &  1        & 2         \\
x_2 & 1       &0          & 2         &  1\\
x_3 & 1       &1          & 1         & 1
\end{array} \right]}
\mbox{ so for example }L_{01}= \det 
{\small \left[ \begin{array}{*{3}c}
x_1        &  1        & 2         \\
x_2         & 1         &  1\\
x_3          & 1         & 1
\end{array} \right]} = x_2-x_3,
\]
which is the equation of the line $L_{01}$ connecting ${\bf v}$ and ${\bf e}_2$,
as expected. Note that combining Theorem~\ref{Aroot} with Terao's theorem yields
\end{exm}
\begin{cor}
$C^r(\widehat{AS(\Delta_n)})$ is a free $S$-module, and Conjecture 1.1 holds.
\end{cor}
\noindent {\bf Closing Remarks}:  Macaulay2 computations were
essential to our work; we are investigating extensions of the 
results here to other types. The Macaulay2 package is available at
{\tt http://www.math.uiuc.edu/Macaulay2}.
\renewcommand{\baselinestretch}{1.0}
\small\normalsize 

\bibliographystyle{amsalpha}

\end{document}